\documentclass[11pt]{article}
\usepackage{epsfig,enumerate}
\usepackage{amssymb}
\usepackage{latexsym}
\usepackage{epic}
\usepackage{eepic}
\usepackage{amsmath}
\usepackage[utf8]{inputenc}
\usepackage{amsthm}
\newcommand{\E}[1]{\textbf{E}\left[#1\right]}
\newcommand{\Prob}[1]{\mbox{Pr}\left\{#1\right\}}
\newcommand{\Probtext}[1]{\mbox{Pr}\left\{\mbox{#1}\right\}}
\newtheorem{thm}{Theorem}
\newtheorem{lem}[thm]{Lemma}
\newlength{\parindentsave}
\newcommand{\proofend}{\hspace*{\fill}\mbox{$\Box$}}
\usepackage{fancyhdr}
\title{On random $k$-out subgraphs of large graphs}

\author{Alan Frieze\thanks{Research supported in part by NSF grant ccf1013110} \ \ \  Tony Johansson\\ Department of Mathematical Sciences\\Carnegie Mellon University\\Pittsburgh PA 15213\\U.S.A.}
\def\ai{A_{\pmb{\iota}}}
\def\bi{B_{\pmb{\iota}}}
\def\ci{C_{\pmb{\iota}}}
\def\ti{T_{\pmb{\iota}}}
\def\si{S_{\pmb{\iota}}}
\def\ri{R_{\pmb{\iota}}}
\def\pathi{P_{\pmb{\iota}}}

\def\S{\Sigma}

\def\a{\alpha}

\def\d{\delta}
\def\D{\Delta}
\def\e{\varepsilon}
\def\f{\phi}
\def\F{\Phi}

\def\G{\Gamma}
\def\k{\kappa}

\def\r{\rho}

\def\om{\omega}

\def\w{\omega}

\def\cD{{\cal D}}
\def\cA{{\cal A}}

\newcommand{\ignore}[1]{}

\newcommand{\rdown}[1]{\mbox{$\left\lfloor #1 \right\rfloor$}}

\def\DFS{\D\F\S}

\def\cE{\mathcal{E}}

\def\cG{{\cal G}}

\newcommand{\brac}[1]{\left( #1 \right)}

\newcommand{\expect}{\operatorname{\bf E}}
\def\E{\expect}
\renewcommand{\Pr}{\operatorname{\bf Pr}}
\newcommand\bfrac[2]{\left[\frac{#1}{#2}\right]}

\def\cH{{\cal H}}
\newcommand{\set}[1]{\left\{#1\right\}}
\begin{document}

\maketitle
\begin{abstract}
We consider random subgraphs of a fixed graph $G=(V,E)$ with large minimum degree. We fix a positive integer $k$ and let $G_k$ be the random subgraph where each $v\in V$ independently chooses $k$ random neighbors, making $kn$ edges in all. When the minimum degree $\d(G)\geq (\frac12+\e)n,\,n=|V|$ then $G_k$ is $k$-connected w.h.p. for $k=O(1)$; Hamiltonian for $k$ sufficiently large. When $\d(G) \geq m$, then $G_k$ has a cycle of length $(1-\e)m$ for $k\geq k_\e$. By w.h.p. we mean that the probability of non-occurrence can be bounded by a function $\f(n)$ (or $\f(m)$) where $\lim_{n\to\infty}\f(n)=0$.
\end{abstract}
\section{Introduction}
The study of random graphs since the seminal paper of Erd\H{o}s and R\'enyi \cite{ER} has by and large been restricted to analysing random subgraphs of the complete graph. This is not of course completely true. There has been a lot of research on random subgraphs of the hypercube and grids (percolation). There has been less research on random subgraphs of arbitrary graphs $G$, perhaps with some simple properties.

In this vain, the recent result of Krivelevich, Lee and Sudakov \cite{KLS} brings a refreshing new dimension. They start with an arbitrary graph $G$ which they assume has minimum degree at least $k$. For $0\leq p\leq 1$ we let $G_p$ be the random subgraph of $G$ obtained by independently keeping each edge of $G$ with probability $p$. Their main result is that if $p=\om/k$ then $G_p$ has a cycle of length $(1-o_k(1))k$ with probability $1-o_k(1)$. Here $o_k(1)$ is a function of $k$ that tends to zero as $k\to\infty$. Riordan \cite{Rio} gave a much simpler proof of this result. Krivelevich and Samotij \cite{KS} proved the existence of long cycles for the case where $p\geq \frac{1+\e}{k}$ and $G$ is $\cH$-free for some fixed set of graphs $\cH$. Frieze and Krivelevich \cite{FK} showed that $G_p$ is non-planar with probability $1-o_k(1)$ when $p\geq \frac{1+\e}{k}$ and $G$ has minimum degree at least $k$. In related works, Krivelevich, Lee and Sudakov \cite{KLS1} considered a random subgraph of a ``Dirac Graph'' i.e. a graph with $n$ vertices and minimum degree at least $n/2$. They showed that if $p\geq \frac{C\log n}{n}$ for suffficently large $n$ then $G_p$ is Hamiltonian with probability $1-o_n(1)$.

The results cited above can be considered to be generalisations of
classical results on the random graph $G_{n,p}$, which in the above
notation would be $(K_n)_p$. In this paper we will consider 
generalising another model of a random graph that we will call
$K_n(k-out)$. This has vertex set $V=[n]=\set{1,2,\ldots,n}$ and each
$v\in V$ independently chooses $k$ random vertices as neighbors. Thus
this graph has $kn$ edges and average degree $2k$. This model in a
bipartite form where the two parts of the partition restricted their
choices to the opposing half was first considered by Walkup \cite{W}
in the context of perfect matchings. He showed that $k\geq 2$ was
sufficient for bipartite $K_{n,n}(k-out)$ to contain a perfect
matching. Matchings in $K_n(k-out)$ were considered by Shamir and
Upfal \cite{SU} who showed that $K_n(5-out)$ has a perfect matching
w.h.p., i.e. with probability $1-o(1)$ as $n\to\infty$. Later, Frieze
\cite{F} showed that $K_n(2-out)$ has a perfect matching w.h.p. Fenner
and Frieze \cite{FF} had earlier shown that $K_n(k-out)$ is
$k$-connected w.h.p. for $k\geq 2$. After several weaker results,
Bohman and Frieze \cite{BF} proved that $K_n(3-out)$ is Hamiltonian
w.h.p. To generalise these results and replace $K_n$ by an arbitrary
graph $G$ we will define $G(k-out)$ as follows: We have a fixed graph
$G=(V,E)$ and each $v\in V$ independently chooses $k$ random
neighbors, from its neighbors in $G$. It will be convenient sometimes
to assume that each $v$ makes its choices with replacement and
sometimes not. It has no real bearing on the results obtained and we
will indicate our choice here. To avoid cumbersome notation, we will from now on assume that $G$ has $n$ vertices and we will refer to $G(k-out)$ as $G_k$. We implicitly consider $G$ to be one of a sequence of larger and larger graphs with $n\to\infty$. We will say that events occur w.h.p. if their probability of non-occurrence can be bounded by a function that tends to zero as $n\to \infty$. 

For a vertex $v\in V$ we let $d_G(v)$ denotes its degree in $G$. Then
we let $\d(G)=\min_{v\in V}\d_G(v)$. We will first consider what we
call Strong Dirac Graphs (SDG) viz graphs with $\d(G)\geq \brac{\frac12+\e}n$ where $\e$ is an arbitrary positive constant.
\begin{thm}\label{th1}
Suppose that $G$ is an SDG. Then w.h.p. $G_k$ is $k$-connected for $2\leq k=O(1)$.
\end{thm}
\begin{thm}\label{th2}
Suppose that $G$ is an SDG. Then there exists a constant $k_\e$ such that if $k\geq k_\e$ then $G_k$ is Hamiltonian.
\end{thm}
Note that we need $\e>0$ in order to prove these results. Consider for
example the case where $G$ consists of two copies of $K_{n/2}$ plus a
perfect matching $M$ between the copies. In this case there is a
probability greater than or equal to $\brac{1-\frac{2k}n}^{n/2}\sim e^{-k}$ that no edge of $M$ will occur in $G_k$.

We will next turn to graphs with large minimum degree.
\begin{thm}\label{th3}
Suppose that $G$ has minimum degree $m$ where $m\to\infty$ with $n$. For every $\e>0$ there exists a constant $k_\e$ such that if $k\geq k_\e$ then w.h.p. $G_k$ contains a path of length $(1-\e)m$.
\end{thm}
Using this theorem as a basis, we strengthen it and prove the existence of long cycles.
\begin{thm}\label{th4}
Suppose that $G$ has minimum degree $m$ where $m\to\infty$ with $n$. For every $\e>0$ there exists a constant $k_\e$ such that if $k \geq k_\e$ then w.h.p. $G_k$ contains a cycle of length $(1 - \e)m$.
\end{thm}

We finally note that in a recent paper, Frieze, Goyal, Rademacher and Vempala \cite{FGRV} have shown that $G_k$ is useful in the construction of sparse subgraphs with expansion properties that mirror those of the host graph $G$. 
\section{Connectivity: Proof of Theorem \ref{th1}}\label{conn}
In this section we will assume that each vertex makes its choices with
replacement. Let $G = (V, E)$ be an SDG. Let $c = 1/(8e)$. We need the following lemma.

\begin{lem}\label{lem1} 
Let $G$ be an SDG and let $C=12/\e$. Then w.h.p. there exists a set $L \subseteq V$, where $|L| \leq C\log n$, such that each pair of vertices $u,v\in V \setminus L$ have at least $12\log n$ common neighbors in $L$.
\end{lem}

\begin{proof} Define $L_p \subseteq V$ by including each $v \in V$ in $L_p$ with probability $p = C\log n/n$. Since $\delta(G) \geq (1/2+\varepsilon)n$, each pair of vertices in $G$ has at least $2\varepsilon n$ common neighbors in $G$. Hence, the number of common neighbors in $L_p$ for any pair of vertices in $V \setminus L_p$ is bounded from below by a $\mbox{Bin}(2\e n, p)$ random variable. 
\begin{eqnarray*}
&&\Probtext{$\exists u,v \in V \setminus L_p$ with less than $12\log n$ common neighbors in $L$} \\
&\leq&n^2 \Prob{\mbox{Bin}(2\varepsilon n, p) \leq 12\log n} \\
&=& n^2 \Prob{\mbox{Bin}(2\varepsilon n, p) \leq \e np}  \\
&\leq& n^2e^{-\e np/8} \\
&=& o(1).
\end{eqnarray*}
Since the expected size of $L_p$ is $C\log n$, there exists a set $L$, $|L| \leq C\log n$, with the desired property.
\end{proof}

Let $L$ be a set as provided by the previous lemma, and let $G_k'$ denote the subgraph of $G_k$ induced by $V \setminus L$.

\begin{lem}\label{lem2} 
Let $c=1/(8e)$. Then w.h.p. all components of $G_k'$ are of size at least $cn$. Furthermore, removing any set of $k-1$ vertices from $G_k'$ produces a graph consisting entirely of components of size at least $cn$, and isolated vertices.
\label{lem:giant} \end{lem}

\begin{proof} 
We first show that w.h.p. $G_k'$ contains no isolated vertex. The probability of $G_k'$ containing an isolated vertex is bounded by
$$\Probtext{$\exists v \in V \setminus L$ which chooses neighbors in $L$ only}\leq n \left[\frac{C\log n}{\frac{1}{2}n} \right]^k=o(1),$$
where $L$ and $C$ are as in Lemma \ref{lem1}.

We now consider the existence of small non-trivial components $S$ after the removal of at most $k-1$ vertices $A$. Then,
\begin{align*}
&\Probtext{$\exists S, A$, $2\leq |S|\leq cn$, $|A| = k-1$, such that $S$ only chooses neighbors in $S\cup L \cup A$} \\
&\leq \sum_{l=2}^{cn} \sum_{|S| = l} \sum_{|A| = k-1}  \left[\frac{l + k-2+C\log n}{\left(\frac{1}{2} +\varepsilon\right)n}\right]^{lk} \\
&\leq \sum_{l=2}^{cn} \binom{n}{l} \binom{n-l}{k-1} \left[\frac{l + C\log n}{\frac{1}{2}n}\right]^{lk} \\
&\leq \sum_{l=2}^{cn} \left(\frac{ne}{l}\right)^ln^{k-1}\left[\frac{l + C\log n}{\frac{1}{2}n}\right]^{lk} \\
&= 2^ke\sum_{l=2}^{cn} \left[\frac{2^k e(l+C\log n)^k}{n^{k-1}l}\right]^{l-1} \frac{(l+C\log n)^k}{l}\\
&\leq 2^ke\brac{\sum_{l=2}^{\log^2n}\bfrac{\log^{3k}n}{n^{k-1}}^{l-1} +\sum_{l=\log^2n}^{cn}(4ec^{k-1})^{l-1}n^k}\\
&=o(1).
\end{align*}
\end{proof}

This proves that w.h.p. $G_k'$ consists of $r \leq 1/c$ components $J_1, J_2, ..., J_r$ and that removing any $k-1$ vertices will only leave isolated vertices and components of size at least $cn$.

\begin{lem}
W.h.p., for any $i \neq j$, there exist $k$ node-disjoint paths (of length 2) from $J_i$ to $J_j$ in $G_k$.
\label{lem:connect}
\end{lem}

\begin{proof}
Let $X$ be the number of vertices in $L$ which pick at least one neighbor in $J_1$ and at least one in $J_2$. 
Furthermore, let $X_{uvw}$ be the indicator variable for $w\in L$ picking $u\in J_1$ and $v\in J_2$ as its neighbors. Note that these variables are independent of $G_k'$.
Let $c=1/(8e)$ as in Lemma \ref{lem2} and let $C=12/\e$ as in Lemma \ref{lem1}. For $w\in L$ we let
$$X_w=\sum_{\substack{(u,v)\in J_1\times J_2\\w\in N_G(J_1)\cap N_G(J_2)}}X_{uvw}.$$
These are independent random variables with values in $\set{0,1,\ldots,k}$.
Let $X=\sum_{w\in L}X_w$. Then,
\begin{eqnarray*}
\E{X} &=& \sum_{u\in J_1} \sum_{v\in J_2} \sum_{\substack{w\in L\\w\in N(J_1)\cap N(J_2)}} \E{X_{uvw}} \\
&=& \sum_{u\in J_1} \sum_{v\in J_2} \sum_{\substack{w\in L\\w\in N(J_1)\cap N(J_2)}} \brac{1-\brac{1-\frac{1}{d_G(u)}}^k}\brac{1-\brac{1-\frac{1}{d_G(v)}}^k}\\
&\geq& \sum_{u\in J_1} \sum_{v\in J_2} \sum_{\substack{w\in L\\w\in N(J_1)\cap N(J_2)}} \frac{1}{n^2}\\
&\geq& \frac{(cn)^2 12\log n}{n^2} \\
&=& 12c^2  \log n.
\end{eqnarray*}
Applying Hoeffdings inequality we get
\begin{equation}\label{EQ1}
\Prob{X\leq k} \leq \Prob{X\leq \frac{\E{X}}{2}} 
\leq \exp \left\{-\frac{(\E{X})^2}{2k^2|L|}\right\} 
=o(1).
\end{equation}
Now for $w_1\neq w_2\in L$ let $\cE(w_1,w_2)$ be the event that $w_1,w_2$ make a common choice. Then
\begin{equation}\label{EQ2}
\Prob{\exists w_1,w_2:\cE(w_1,w_2)}=O\bfrac{k^2\log^2n}{n}=o(1).
\end{equation}

Equations \eqref{EQ1} and \eqref{EQ2} together show that w.h.p., there are $k$ node-disjoint paths from $J_1$ to $J_2$. Since the number of giant components is bounded by a constant, this is true for all pairs $J_i, J_j$ w.h.p.
\end{proof}

We can complete the proof of Theorem \ref{th1}.
Suppose we remove $l$ vertices from $L$ and $k-1-l$ vertices from the remainder of $G$. We know from Lemma \ref{lem1} that $V\setminus L$ induces components $C_1,C_2,\ldots,C_r$ of size at least $cn$. There cannot be any isolated verticesin $V\setminus L$ as $G_k$ has minimum degree at least $k$. Lemma \ref{lem2} implies that $r=1$ and that every vertex in $L$ is adjacent to $C_1$.
\proofend

\section{Hamilton cycles: Proof of Theorem \ref{th2}}
In this section we will assume that each vertex makes its choices
without replacement.
Let $G$ be a graph with $\delta(G) \geq (1/2+\varepsilon)n$, and let
$k$ be a positive integer. 

Let $\cD(k,n)=D_1, D_2,...,D_M$ be the $M = \prod_{v\in V}\binom{d_G(v)}{k}
\leq \binom{n-1}{k}^n$ directed graphs obtained by letting each vertex
$x$ of $G$ choose $k$ $G$-neighbors $y_1,...,y_k$, and including in
$D_i$ the $k$ arcs $(x,y_i)$. Define $\vec{N}_i(x) = \{y_1,...,y_k\}$
and for $S \subseteq V$ let $\vec{N}_i(S) = \bigcup_{x\in S}
\vec{N}_i(x) \setminus S$. For a digraph $D$ we let $G(D)$ denote the
graph obtained from $D$ by ignoring orientation and coalescing
multiple edges, if necessary. We let $\G_i=G(D_i)$ for $i=1,2,\ldots,M$. Let $\cG(k,n) =
\{\G_1,\G_2,...,\G_M\}$ be the set of $k$-out graphs on $G$.
Below, when we say that $D_i$ is Hamiltonian we actually mean that
$\G_i$ is Hamiltonian. (It will occasionally enable more succint statements).

For each $D_i$, let $D_{i1}, D_{i2},...,D_{i\k}$ be the $\k =k^n$
different edge-colorings of $D_i$ in which each vertex has $k-1$
outgoing green edges and one outgoing blue edge. Define $\G_{ij}$ to be
the colored (multi)graph obtained by ignoring the orientation of edges in
$D_{ij}$. Let $\G^g_{ij}$ be the subgraph induced by green edges.

$\vec{N}(S)$ refers to
$\vec{N}_i(S)$ when $i$ is chosen uniformly from $[M]$, as it will be
for $G_k$. 
\begin{lem}
Let $k \geq 5$. There exists an $\alpha > 0$ such that the following holds w.h.p.: for any set $S \subseteq V$ of size $|S| \leq \alpha n$, $|\vec{N}(S)| \geq 3|S|$.
\label{lem:expand}
\end{lem}

\begin{proof}
The claim fails if there exists an $S$ with $|S| \leq \alpha n$ such that there exists a $T$, $|T| = 3|S|-1$ such that $\vec{N}(S) \subseteq T$. The probability of this is bounded from above by
\begin{eqnarray*}
&& \sum_{l=1}^{\alpha n} \binom{n}{l}\binom{n-l}{3l-1} \prod_{v\in S} \left[\binom{4l-2}{k} \middle /\binom{d_G(v)}{k}\right] \\
%&\leq& \sum_{l=1}^{\alpha n}\left(\frac{ne}{l}\right)^l \left(\frac{(n-l)e}{3l-1}\right)^{3l-1} \left[\binom{4l-2}{k} %\middle / \binom{n/2}{k}\right]^l \\
&\leq& \sum_{l=1}^{\alpha n} \left(\frac{ne}{l}\right)^l \left(\frac{ne}{3l-1}\right)^{3l-1} \left[\frac{4le}{n/2}\right]^{kl} \\
&\leq& \sum_{l=1}^{\alpha n}\left[e^4 (8e)^k
  \left(\frac{l}{n}\right)^{k-4}\right]^l\\
&=&o(1),
\end{eqnarray*}
for $\alpha = 2^{-16}e^{-9}$.
\end{proof}

We say that a digraph $D_i$ \textit{expands} if $|\vec{N}_i(S)|\geq 3|S|$
whenever $|S| \leq \alpha n$, $\alpha = 2^{-16}e^{-9}$. Since almost
all $D_i$ expand, we need only prove that an expanding $D_i$ almost
always gives rise to a Hamiltonian $\G_i$. Write $\mathcal{D}'(k,n)$
for the set of expanding digraphs in $\mathcal{D}(k,n)$
and let $\cG'(k,n)=\set{\G_i:D_i\in \cD'(k,n)}$.

Let $H$ be any graph, and suppose $P = (v_1,...,v_k)$ is a longest
path in $H$. If $t \neq 1, k-1$ and $\{v_k, v_t\} \in E(H)$, then
$P'=(v_1,...,v_t, v_k, v_{k-1},...,v_{t+1})$ is also a longest path of
$H$. Repeating this rotation for $P$ and all paths created in the
process, keeping the endpoint $v_1$ fixed, we obtain a set $EP(v_1)$
of other endpoints. 

For $S\subseteq V(H)$ we let $N_H(S)=\set{w\notin S:\exists v\in S\
  s.t.\ vw\in E(H)}$.
\begin{lem} [P\'osa]
For any endpoint $x$ of any longest path in any graph $H$, $|N_H(EP(x))| \leq 2 |EP(x)| - 1$.
\label{lem:posa}
\end{lem}
We say that a graph expands if  $|N_H(S)|\geq 2|S|$ whenever $|S| \leq \alpha n$, assuming $|V(H)|=n$.
\begin{lem}
Consider a green subgraph $\G^g_{ij}$. W.h.p.,
there exists an $\alpha > 0$ such that for every longest path $P$ in $\G^g_{ij}$ and endpoint $x$ of $P$, $|EP(x)| > \alpha n$.
\label{lem:greenexpand}
\end{lem}

\begin{proof}
Let $H=\G^g_{ij}$. We argue that if $D_i$ expands then so does $H$. For $S\subseteq V$ let $N_{i,j}(S)=N_H(S)$. If $|\vec{N}_i(S)| \geq 3 |S|$, then $|N_{i,j}(S)| \geq 2|S|$, since each
vertex of $S$ picks at most one blue edge outside of $S$. Thus $H$ expands. In
particular, Lemma \ref{lem:expand} implies that if $|S| \leq \alpha
n$, then $|\vec{N}(S)| \geq 3|S|$ and hence $|N_{i,j}(S)| \geq
2|S|$. By Lemma \ref{lem:posa}, this implies that $|EP(x)| > \alpha n$
for any longest path $P$ and endpoint $x$.
\end{proof}

Let $M_1$ be the number of expanding digraphs $D_i$ among $D_1,...,D_M$ for
which $\G_i$ is not Hamiltonian. We aim to show that $M_1/M
\rightarrow 0$ as $n$ tends to infinity. W.l.o.g. suppose
$\mathcal{N}(k,n) = \{D_1,...,D_{M_1}\}$ are the expanding digraphs
which are not Hamiltonian. Define $a_{ij}$ to be $1$ if
$\G^g_{ij}$ contains a longest path of $\G_{ij}$, $1\leq i \leq M_1$, and $0$ otherwise.

\begin{lem}
For $1 \leq i\leq M_1$, we have $\sum_{j=1}^N a_{ij} \geq (k-1)^n$.
\end{lem}

\begin{proof}
Fix $1 \leq i\leq M_1$ and a longest path $P$ of $\G_i$. Uniformly picking one of $D_{i1},...,D_{iN}$, we have
\begin{eqnarray*}
\Prob{a_{ij}=1} &\geq& \Probtext{$E(P) \subseteq E(\G^g_{ij})$} \\
&\geq& \left(1-\frac{1}{k}\right)^{|E(P)|} \\
&\geq& \left(1-\frac{1}{k}\right)^n
\end{eqnarray*}
The lemma follows from the fact that there are $k^n$ colorings of $D_i$.
\end{proof}

Let $\Delta\in \mathcal{N}(k-1,n)$ be expanding and non-Hamiltonian
and for the purposes of exposition consider its edges to be colored
green. Let $D \in \mathcal{D}(k,n)$ be the random digraph obtained by
letting each vertex of $\Delta$ choose another edge, which will be colored blue. Let $\overline{B_\Delta}$ be the event that $D$ contains a path longer than the longest path of $\Delta$ or if $D$ is Hamiltonian, implying that $a_{ij} = 0$. Since $G(D)$ is connected w.h.p., this event is the union of the events 
$$D \text{ has an edge between the endpoints of a longest path of }G(\Delta)$$
 and 
$$D \text{ has an edge from an endpoint of a longest path of $\Delta$ to the complement of the path}.$$ 
Let $B_\Delta$ be the complement of $\overline{B_\Delta}$ and for Hamiltionian $\Delta$ let $B_\Delta = \emptyset$.

Let $N_\Delta$ be the number of $i, j$ such that $\Delta_{ij}  = \Delta$. We have
\begin{equation}
\sum_{i,j:\Delta_{ij} = \Delta} a_{ij} = N_\Delta \Prob{B_\Delta}
\end{equation}
The number of non-Hamiltonian graphs is bounded by
\begin{eqnarray}
M_1 &\leq& \sum_{i=1}^M \sum_{j=1}^N \frac{a_{ij}}{(k-1)^n} \nonumber \\
&\leq& \frac{\sum_{\Delta} N_\Delta \Prob{B_\Delta}}{(k-1)^n} \nonumber \\
&\leq& \frac{Mk^n \max_\Delta \Prob{B_\Delta}}{(k-1)^n} \nonumber \\
&=& M \frac{\max_\Delta \Prob{B_\Delta}}{(1-1/k)^n}
\label{eq:hambound}
\end{eqnarray}

Fix a $\Delta\in \mathcal{N}(k-1,n)$. Let $EP$ be the set of vertices which are endpoints of a longest path of $\Delta$. For $x\in EP$, say $x$ is of Type I if $x$ has at least $\varepsilon n/2$ neighbors outside $P$, and Type II otherwise. Let $E_1$ be the set of Type I endpoints, and $E_2$ the set of Type II endpoints.

Partition the set of expanding green graphs by
\begin{equation}
\mathcal{D}'(k-1,n) = \mathcal{H}(k-1,n) \cup \mathcal{N}_1(k-1,n)\cup \mathcal{N}_2(k-1,n)
\end{equation}
where $\mathcal{H}(k-1,n)$ is the set of Hamiltonian graphs, $\mathcal{N}_1(k-1,n)$ the set of non-Hamiltonian graphs with $|E_1|\geq \alpha n/2$ and $\mathcal{N}_2(k-1,n)$ the set of non-Hamiltonian graphs with $|E_1| < \alpha n/2$.
\begin{lem}
For $\Delta \in \mathcal{N}_1(k-1,n)$, $\Prob{B_\Delta} \leq e^{-\varepsilon\alpha n/4}$.
\label{lem:n1bound}
\end{lem}
\begin{proof}
Let each $x \in E_1$ choose a neighbor $y(x)$. The event $B_\Delta$ is included in the event $\{\forall x \in E_1 : y(x) \in P\}$. We have
\begin{eqnarray*}
\Prob{B_\Delta} &\leq& \Prob{\forall x \in E_1: y(x) \in P} \\
&=& \prod_{x \in E_1}\frac{d_{{P}}(x)}{d_G(x)}\\
%&\leq& \left(1 - \frac{(\varepsilon/2)n}{n}\right)^{\alpha n/2} \\
&\leq& \left(1 - \frac{\varepsilon}{2}\right)^{\alpha n/2}
\end{eqnarray*}
where $d_{P}(x)$ denotes the number of neighbors of $x$ inside $P$.
\end{proof}

\begin{lem}
For $\Delta \in \mathcal{N}_2(k-1,n)$, $\Prob{B_\Delta} \leq e^{-\varepsilon\alpha^2n/129}$.
\label{lem:n2bound}
\end{lem}

\begin{proof}
Let $X \subseteq E_2$ be a set of $\alpha n/4$ Type II endpoints. For each $x \in X$, let $A(x)$ be the set of Type II vertices $y \notin X$ such that a path from $x$ to $y$ in $\Delta$ can be obtained from $P$ by a sequence of rotations with $x$ fixed. By Lemma \ref{lem:greenexpand} we have $|A(x)| \geq \alpha n/4$ for each $x$, since $A(x) = EP(x) \setminus (E_1 \cup X)$.

Let $P_{x,y}$ be a path with endpoints $x \in X, y \in A(x)$ obtained from $P$ by rotations with $x$ fixed, and label the vertices on $P_{x,y}$ by $x=z_0, z_1,...,z_l = y$. Suppose $y$ chooses some $z_i$ on the path with its blue edge. If $\{z_{i+1}, x\} \in E(G)$, let $B_y(x) = \{z_{i+1}\}$. Write $v(y)$ for $z_{i+1}$. If $\{z_ {i+1}, x\} \notin E(G)$, or if $y$ chooses a vertex outside $P$, let $B_y(x) = \emptyset$.

\begin{figure}[h]
\centering
\includegraphics{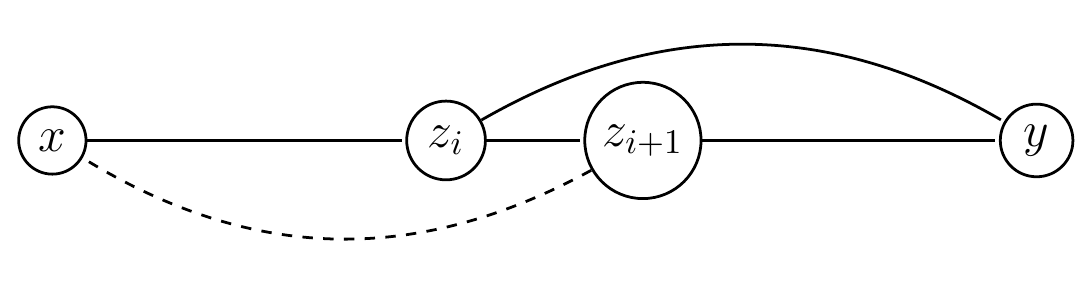}
\caption{Suppose $y$ chooses $z_i$. The vertex $z_{i+1}$ is included in $B(x)$ if and only if $\{x, z_{i+1}\} \in E(G)$.}
\end{figure}
There will be at least $2\brac{\frac12+\frac{\e}2}n-n=\e n$ choices for $i$ for which $\{x,z_{i+1}\} \in E(G)$. 
Let $Y_x$ be the number of $y \in A(x)$ such that $B_y(x)$ is nonempty. This variable is bounded stochastically from below by a binomial $\mbox{Bin}(\alpha n/4, \varepsilon)$ variable, and by a Chernoff bound we have that
\begin{equation}
\Prob{\exists x: Y_x \leq \frac{\e\alpha n}{8}} \leq n\exp\left\{-\frac{\varepsilon \alpha n}{32}\right\}
\end{equation}

Define $B(x) = \bigcup_{y \in A(x)} B_y(x)$. Conditional on $Y_x \geq \varepsilon \alpha n/8$ for all $x\in X$, let $y_1,y_2,...,y_r$ be $r = \varepsilon \alpha n/8$ vertices whose choice produces a nonempty $B_y(x)$. Let $Z_x = |B(x)|$, and for $i=1,...,r$ define $Z_i$ to be $1$ if $v(y_i)$ is distinct from $v(y_1),...,v(y_{i-1})$ and $0$ otherwise. We have $Z_x = \sum_{i=1}^r Z_i$, and each $Z_i$ is bounded from below by a Bernoulli variable with parameter $1-\alpha/8$. To see this, note that $y_i$ has at least $\varepsilon n$ choices resulting in a nonempty $B_{y_i}(x)$ since $x$ and $y_i$ are of Type II, so
\begin{equation}
\Prob{\exists j< i: v(y_j) = v(y_i)} \leq \frac{i-1}{\varepsilon n} \leq \frac{\varepsilon \alpha n/8}{\varepsilon n} = \frac{\alpha}{8}
\end{equation}
Since $\alpha/8 < 1/2$, $Z_x$ is bounded stochastically from below by a binomial $\mbox{Bin}(\varepsilon \alpha n/8, 1/2)$ variable, and so
\begin{equation}
\Prob{\exists x: Z_x < \frac{\e\alpha n}{32}} \leq n \exp\left\{-\frac{\varepsilon\alpha n}{128}\right\}
\end{equation}
Each $x$ for which $Z_x \geq \varepsilon \alpha n/32$ will choose a vertex in $B(x)$ with probability
\begin{equation}
\frac{|B(x)|}{d_G(x)} \geq \frac{\varepsilon \alpha n /32}{n} = \frac{\varepsilon \alpha}{32}
\end{equation}
Hence we have
\begin{equation}
\Prob{B_\Delta} \leq \left(1-\frac{\varepsilon \alpha}{32}\right)^{\alpha n/4} + n\exp\left\{-\frac{\varepsilon\alpha n}{32}\right\} + n\exp\left\{-\frac{\varepsilon \alpha n}{128}\right\}\leq e^{-\e\a^2n/129}.
\end{equation}
\end{proof}

We can now complete the proof of Theorem \ref{th2}. From Lemmas \ref{lem:n1bound} and \ref{lem:n2bound} we have
$$\Prob{B_\Delta} \leq \max\set{e^{-\e\a n/4},e^{-\e\a^2n/129},0}.$$
Going back to \eqref{eq:hambound} with $k = C/\varepsilon $ we have
\begin{eqnarray*}
\Probtext{$G_k$ is non-Hamiltonian} &=& \frac{M_1}{M} \\
&\leq& \frac{\max_\Delta \Prob{B_\Delta}}{(1-1/k)^n} \\
&=& \left[\frac{e^{-\e\a^2/129}}{1-\varepsilon/C}\right]^n \\
&\leq& \exp\left\{-\varepsilon\left(\frac{\alpha^2}{129} - \frac{2}{C}\right)n\right\}\\
&=&o(1),
\end{eqnarray*}
for $C = 259/ \alpha^2$.
\proofend

\section{Long Paths: Proof of Theorem \ref{th3}}
Let $D_k$ denote the directed graph with out-degree $k$ defined the vertex choices. Consider a Depth First Search (DFS) of $D_k$ where we construct $D_k$ as we go. At all times we keep a stack $U$ of vertices which have been visited, but for which we have chosen fewer than $k$ out-edges. $T$ denotes the set of vertices that have not been visited by DFS. Each step of the algorithm begins with the top vertex $u$ of $U$ choosing one new out-edge. If the other end of the edge $v$ lies in $T$ (we call this a \textit{hit}), we move $u$ from $T$ to the top of $U$.

When DFS returns to $v\in U$ and at this time $v$ has chosen all of its $k$ out-edges, we move $v$ from $U$ to $S$. In this way we partition $V$ into
\begin{itemize}
\item[$S$ - ] Vertices that have chosen all $k$ of its out-edges.
\item[$U$ - ] Vertices that have been visited but have chosen fewer than $k$ edges.
\item[$T$ - ] Unvisited vertices.
\end{itemize}
Key facts: Let $h$ denote the number of hits at any time and let $\k$ denote the number of times we have re-started the search i.e. selected a vertex in $T$ after the stack $S$ empties. 
\begin{enumerate}[{\bf P1}]
\item $|S\cup U|$ increases by $1$ for each hit, so $|S\cup U| \geq h$.
\item More specifically, $|S \cup U| = h + \k -1$.
\item $S\cup U$ contains a path which contains all of $U$ at all times.
\end{enumerate}

The goal will be to prove that $|U| \geq (1-2\varepsilon)m$ at some point of the search, where $\e$ is some arbitrarily small positive constant.

\begin{lem}
After $\varepsilon k m$ steps, i.e. after $\varepsilon k m$ edges have been chosen in total, the number of hits $h\geq (1-\varepsilon) m$ w.h.p.
\label{lem:hits}
\end{lem}

\begin{proof}
Since $\delta(G_k)\geq k$, each tree component of $G_k$ has at least $k$ vertices, and at least $k^2$ edges must be chosen in order to complete the search of the component. Hence, after $\varepsilon k m$ edges have been chosen, at most $\varepsilon km/k^2 \leq \varepsilon m/2$ tree components have been found. This means that if $h\leq(1-\varepsilon)m$ after $\varepsilon km$ edges have been sent out, then {\bf P2} implies that $|S\cup U| \leq (1-\varepsilon/2)m$. 

So if $h\leq (1-\varepsilon)m$ each edge chosen by the top vertex $u$ has probability at least $\frac{d(u)-|S\cup U|}{d(u)}\geq \e/2$ of making a hit. Hence,
\begin{equation}
\Probtext{$h\leq (1-\varepsilon)m$ after $\varepsilon km$ steps} \leq \Prob{\mbox{Bin}(\varepsilon km, \varepsilon/2) \leq (1-\varepsilon) m}=o(1),
\end{equation}
for $k\geq 2/\e^2$, by the Chernoff bounds.
\end{proof}

We can now complete the proof of Theorem \ref{th3}.
By Lemma \ref{lem:hits}, after $\varepsilon km$ edges have been chosen we have $|S \cup U| \geq (1-\varepsilon)m$ w.h.p. For a vertex to be included in $S$, it must have chosen all of its edges. Hence, $|S| \leq \varepsilon km/k = \varepsilon m$, and we have $|U| \geq (1-2\varepsilon)m$. Finally observe that $U$ is the set of vertices of a path of $G_k$.
\proofend

\section{Long Cycles: Proof of Theorem \ref{th4}}

Suppose now that we consider $G_{4k}=LR_k\cup DR_k \cup LB_k\cup DB_k$ where each vertex makes $k$ choices each of the colors ``light red", ``dark red", ``light blue" and ``dark blue".  $LR_k, DR_k, LB_k, DB_k$ respectively are the graphs induced by the differently colored choices. We have by Theorem \ref{th3}
that w.h.p. there is a path $P$ of length at least $(1-\e)m$ in the
light red graph $LR_k$. At this point we start using a modification of DFS (denoted by
$\D\F\S$) and the differently colored choices to create a cycle.

We divide the steps into epochs $T_0,T_{00}, T_{01},\ldots$, indexed by binary strings. We stop the search immediately if there is a high chance of finding a cycle of length at least $(1-19\e)m$. If executed, epoch $\ti, \pmb{\iota} = 0***$ will extend the exploration tree by at least $(1-5\e)m$ vertices, unless an unlikely failure occurs. Theorem \ref{th3} provides $T_0$. In the remainder, we will assume $\pmb{\iota} \neq 0$.

Epoch $\ti$ will use light red colors if $i$ has odd length and ends in a $0$, dark red if $i$ has even length and ends in a $0$, light blue if $i$ has odd length and ends in a $1$, and dark blue if $i$ has even length and ends in a $1$. Epochs $T_{\pmb{\iota}0}$ and $T_{\pmb{\iota}1}$ (where $\pmb{\iota}j$ denotes the string obtained by appending $j$ to the end of $\pmb{\iota}$) both start where $\ti$ ends, and this coloring ensures that every vertex discovered in an epoch will initially have no adjacent edges in the color of the epoch.

During epoch $\ti$ we maintain a stack of vertices $\si$.
When discovered, a vertex is placed in one of the three sets $\ai, \bi, \ci$, and simultaneously placed in $\si$ if it is placed in $\ai$. Once placed, the vertex remains in its designated set even if it is removed from $\si$.  Let $d_T(v,w)$ be the length of the unique path in the exploration tree $T$ from $v$ to $w$. We designate the set for $v$ as follows.
\begin{itemize}
\item[$\ai$ -] $v$ has less than $(1-2\e)d(v)$ $G$-neighbors in $T$.
\item[$\bi$ -] $v$ has at least $(1-2\e)d(v)$ $G$-neighbors in $T$, but less than $\e d(v)$ $G$-neighbors $w$ such that $d_T(v,w) \geq (1-19\e)m$.
\item[$\ci$ -] $v$ has at least $(1-2\e)d(v)$ $G$-neighbors in $T$, and at least $\e d(v)$ $G$-neighbors $w$ such that $d_T(v,w) \geq (1-19\e)m$.
\end{itemize}

At the initiation of epoch $\ti$, a previous epoch will provide a set $\ti^0$ of $3 \e m$ vertices, as described below. Starting with $\ai = \bi = \ci = \emptyset$, each vertex of $\ti^0$ is placed in $\ai, \bi$ or $\ci$ according to the rules above. Let $\si = \ai$, ordered with the latest discovered vertex on top.

If at any point during $\ti$ we have $|\bi| = \e m$ or $|\ci| = \e m$, we immediately interrupt $\DFS$ and use the vertices of $\bi$ or $\ci$ to find a cycle, as described below.

An epoch $\ti$ consists of up to $\e km$ steps, and each step begins with a $v \in \ai$ at the top of the stack $\si$. This vertex is called \emph{active}. If $v$ has chosen $k$ neighbors, remove $v$ from the stack and perform the next step. Otherwise, let $v$ randomly pick one neighbor $w$ from $N_G(v)$. If $w\notin T$, then $w$ is assigned to $\ai, \bi$ or $\ci$ as described above. If $w \in \ai$, perform the next step with $w$ at the top of $\si$. If $w \in T$, perform the next step without placing $w$ in $\si$.

The exploration tree $T$ is built by adding to it any vertex found during $\DFS$, along with the edge used to discover the vertex.

Note that unless $|\bi| = \e m$ or $|\ci| = \e m$, we initially have $|\ai| \geq \e m$, guaranteeing that $\e km$ steps may be executed. Epoch $\ti$ \emph{succeeds} and is ended (possibly after fewer than $\e km$ steps) if at some point we have $|\ai| = (1-2\e)m$. If all $\e km$ steps are executed and $|\ai| < m$, the epoch fails.

\begin{lem}
Epoch $\ti$ succeeds with probability at least $1-e^{-\e^2m/8}$, unless $|\bi| = \e m$ or $|\ci| = \e m$ is reached.
\end{lem}

\begin{proof}
An epoch fails if less than $(1-3\e)m$ steps result in the active vertex choosing a neighbor outside $T$. Since the active vertex is always in $\ai$, we have
\begin{equation*}
\Probtext{$\ti$ finishes with $|\ai| < (1-2\e)m$} \leq \Probtext{Bin$(\e km, 2\e) < (1-2\e)m$} \leq e^{-\e^2m/8}
\end{equation*}
for $k \geq 1/2\e^2$, by Hoeffding's inequality. This proves the lemma.
\end{proof}

The epoch produces a tree which is a subtree of $T$. Let $\pathi$ be the longest path of vertices in $\ai$, and let $\ri$ be the set of vertices discovered during $\ti$ which are not in $\pathi$. If the epoch succeeds, $\pathi$ has length at least $(1-6\e)m$, and at most $3\e m$ vertices discovered during $\ti$ are not on the path. Indeed, a vertex is outside $\pathi$ if and only if it is in $\ai$ and has chosen all its $k$ neighbors, or if it is in $\bi \cup \ci$. Thus, the number of vertices not on the path is bounded by
\begin{equation*}
|\ri| \leq \frac{\e km}{k} + |\bi| + |\ci| < 3\e m.
\end{equation*}
If the epoch fails, the path $\pathi$ may be shorter, but $|\ri|$ is still bounded by $3 \e m$.% This implies that for any set $S$, $|S| \geq 3 \e m+2$ that was discovered in one epoch, there exist $v,w \in S \cap \pathi$  such that $d_T(v,w) \geq |S| - 3\e m$.

If $\ti$ succeeds, the epochs $T_{\pmb{\iota}0}$ and $T_{\pmb{\iota}1}$ will be initiated at the end of $\ti$, by letting $T_{\pmb{\iota}0}^0$ and $T_{\pmb{\iota}1}^0$ be the last $3 \e m$ vertices discovered during $\ti$. If $\ti$ fails, $T_{\pmb{\iota}0}$ and $T_{\pmb{\iota}1}$ will not be initiated. The exploration tree $T$ will resemble an unbalanced binary tree, in which each successful epoch gives rise to up to two new epochs. Epochs are ordered after their binary value, so that $T_{\pmb{\iota}_1}$ is initiated before $T_{\pmb{\iota}_2}$ if and only if $\pmb{\iota}_1 < \pmb{\iota}_2$, ordered according to the numerical value of the binary strings.

\begin{lem}
W.h.p., $\DFS$ will discover an epoch $\ti$ having $|\bi| = \e m$ or $|\ci| = \e m$.
\end{lem}

\begin{proof}
Suppose that no epoch ends with $|\bi| = \e m$ or $|\ci| = \e m$. Under this assumption, we may model the exploration as a Galton-Watson branching process, in which a successful $\ti$ gives rise to at least $X_i$ successful epochs, where $X_i = 0$ with probability $e^{-2cm}$, $X_i=1$ with probability $2e^{-cm}(1-e^{-cm})$ and $X_i = 2$ with probability $(1-e^{-cm})^2$. The offspring distribution for this lower bound has generating function
\begin{equation*}
G_m(s) = e^{-2cm} + 2se^{-cm}(1-e^{-cm}) + s^2(1-e^{-cm})^2.
\end{equation*}
Let $s_m$ be the smallest fixed point $G_m(s_m) = s_m$. We have $s_m \rightarrow 0$ as $m\rightarrow \infty$. Hence, the probability that the branching process never expires is at least $1-s_m$, which tends to $1$.

The number of epochs is bounded by a finite number. Hence, the branching process cannot be infinite. This contradiction finishes the proof.
\end{proof}

We may now finish the proof of the theorem. Condition first on $\DFS$
being stopped by an epoch $\ti$ having $|\ci| = \e m$. In this case,
let each $v \in \ci$ choose $k$ neighbors using eges with the epoch's color. Each choice has probability at least $\e$ of finding a cycle of length at least $(1-19\e)m$, by choosing a neighbor $w$ such that $d_T(v,w) \geq (1-19\e)m$. The probability of not finding a cycle of length at least $(1-19\e)m$ is bounded by
\begin{equation*}
(1-\e)^{\e km}\to 0.
\end{equation*}

Now condition on $\DFS$ being stopped by an epoch $\ti$ having $|\bi|
= \e m$. Note that we must have $\pmb{\iota} = \pmb{\iota}'1$ for some
$\pmb{\iota}'$. Indeed, if $\pmb{\iota} = \pmb{\iota'}0$, then any $v$
discovered in $\pmb{\iota}$ must have at least $11\e d(v)$ $G$-neighbors
at distance at least $(1-19\e)m$, at its
time of discovery. If not, and $v\notin \ai$ then it has at most $2\e d(v)$ $G$-neighbors
outside $T$, at most $3\e d(v)+3\e d(v)$ $G$-neighbors in
$R_{\pmb{\iota}}\cup R_{\pmb{\iota}'}$. There are at most
$(1-19\e)d(v)$ $G$-neighbors in $T\setminus(R_{\pmb{\iota}}\cup R_{\pmb{\iota}'})$ at distance less than
$(1-19\e)d(v)$ and so there are at least $11\e d(v)$ $G$-neighbors in
$T$ at distance at least $(1-19\e)d(v)$ from $v$, which implies that
$v\in C_{\pmb{\iota}}$, contradiction. 

Note also that $d(v) \leq 2m$ for any $v \in \bi$. This can be seen as follows: For any $v\in W$ let $\rho_v \in \ti^0$ be the vertex which minimizes $d_T(v, \rho_v)$. Note that we may have $\rho_v = v$. There are at most $|Q|$ $G$-neighbors of $v$ on the path $Q$ from $v$ to $\r_v$. Then note that there are at most $2((1-19\e)m - |Q|)$ $G$-neighbors of $v$ on $T\setminus (Q\cup \ri\cup R_{\pmb{\iota}'}\cup R_{\pmb{\iota}'0})$  that are within $(1-19\e)m$ of $v$. So the maximum number of $w\in N_G(v)\cap T$ such that $d_T(v,w) \leq (1-19\e)m$ is bounded by
\begin{equation}\label{d(v)}
|Q| + 2((1-19\e)m - |Q|) + |\ri| + |R_{\pmb{\iota}'}|+|R_{\pmb{\iota}'0}|  \leq (2-29\e)m
\end{equation}
Equation \eqref{d(v)} then implies that $d(v)\leq (2-29\e)m+3\e d(v)$. 

Since the epoch produces a tree with at most $m$ vertices, using the pigeonhole principle we can choose a $W \subseteq \bi$ such that $|W| = \e^2 m$ and $d_T(v,w) \leq \e m$ for any $v,w\in W$. 

Define an ordering on $T$ by saying that $t_1 \leq t_2$ if $t_1$ was discovered before $t_2$ during $\DFS$, or if $t_1 = t_2$. If $S \subseteq T'$, and $t \leq s$ for all $s \in S$, write $t \leq S$. Similarly define $\geq, >$ and $<$.

\begin{figure}[h]
\centering
\includegraphics[width=120mm]{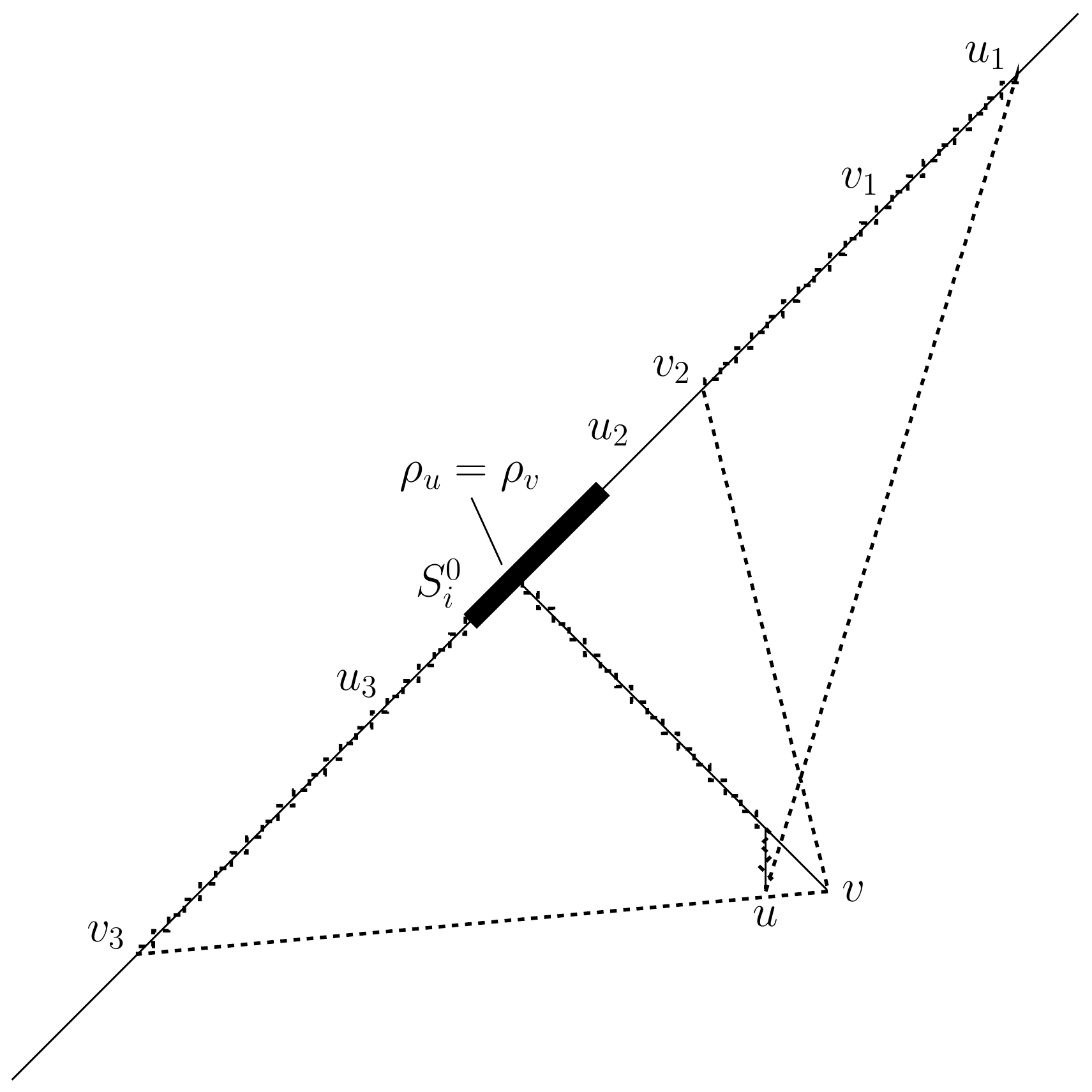}
\caption{Example depiction of cycle found when $|\bi| = \e m$.}
\end{figure}

Let each $v \in W$ choose $k$ neighbors in the color of epoch $\ti$. We say that $v$ is \emph{good} if it chooses $v_1,v_2 \in P_{\pmb{\iota}'}$ and $v_3 \in P_{\pmb{\iota}'0}$ such that
\begin{equation*}
d_T(v_1,v_2) + d_T(v_3,\ti^0) + d_T(\rho_v, v) \geq (1-17\e)m
\end{equation*}
where $d_T(v_3,S) = \min_{s \in S} d_T(v_3,s)$. For each $v\in W$ define $n_0(v) = |N_G(v) \cap \pathi \setminus \ti^0|$, $n_1(v) = |N_G(v) \cap P_{\pmb{\iota}'} \setminus \ti^0|$ and $n_2(v) = |N_G(v) \cap P_{\pmb{\iota}'0} \setminus \ti^0|$. Since $v\in \bi$ we have
\begin{equation*}
n_0(v) + n_1(v) + n_2(v) = |(N_G(v)\cap T) \setminus (R_{\pmb{\iota}'} \cup R_{\pmb{\iota}'0} \cup \ri \cup \ti^0)| \geq (1-14\e)m.
\end{equation*}
Since the $n_0(v) + n_1(v)$ vertices of $N_G(v) \cup \pathi \cup P_{\pmb{\iota}'} \setminus \ti^0$ are on a path, we must have $n_0(v) + n_1(v) \leq (1-16\e)m$, otherwise $v$ has $2\e m \geq \e d(v)$ neighbors at distance at least $(1-18\e)m$, contradicting $v\in\bi$. This implies $n_2(v) \geq 2\e m$. Similarly, $n_1(v) \geq 2\e m$.

Fix a vertex $v\in W$ and define $V_1, V_2 \subseteq (N_G(v) \cap P_{\pmb{\iota}'}) \setminus \ti^0$ and $V_3 \subseteq (N_G(v) \cap P_{\pmb{\iota}'0}) \setminus \ti^0$, $|V_1| = |V_2| = |V_3| = \e m$ as follows. $V_1$ is the set of the first $\e m$ vertices of $N_G(v) \cap P_{\pmb{\iota}'}$ discovered during $\DFS$. $V_2$ is the set of the last $\e m$ vertices of $N_G(v) \cap P_{\pmb{\iota}'}$ discovered before any vertex of $\ti^0$. Lastly, $V_3$ consists of the $\e m$ last vertices discovered in $N_G(v) \cap P_{\pmb{\iota}'0}$. Since $n_1(v)\geq 2 \e m$ and $n_2(v) \geq 2 \e m$, the sets $V_1,V_2,V_3$ exist and are disjoint.

Since $d(v) \leq 2m$, the probability that $v$ chooses $v_1 \in V_1, v_2 \in V_2$ and $v_3 \in V_3$ is at least $(\e/2)^3$. If this happens, we have
\begin{equation*}
d_T(v_1,v_2) + d_T(v_3,\ti^0) + d_T(\rho_v, v) \geq n_1(v) - 2\e m + n_2(v) - \e m + n_3(v) \geq (1-17\e) m.
\end{equation*}
In other words, $v \in W$ is good with probability at least $(\e/2)^3$. Since $|W| = \e^2 m$, w.h.p. there exist two good vertices $u,v \in W$. Since $u, v\notin \pathi$, the shortest path from $\rho_v$ to $v$ does not contain $u$, and the shortest path from $\rho_u$ to $u$ does not contain $v$. Also, by choice of $W$ we have $d_T(\rho_u,u) \geq d_T(\rho_v,v) - 2\e m$. Suppose $u$ and $v$ pick $u_1 \leq u_2 \leq u_3$ and $v_1 \leq v_2 \leq v_3$, and w.l.o.g. suppose $d_T(u_1,v_2) \geq d_T(v_1,v_2)$. The cycle $(u, u_1,...,v_2,v,v_3,...,\rho_u,..., u)$ has length
\begin{eqnarray*}
&& 1 + d_T(u_1,v_2) + 1 + 1 + d_T(v_3,\rho_u) + d_T(\rho_u,u) \\
&\geq& d_T(v_1,v_2) + d_T(v_3,\ti^0) + d_T(\rho_v,v) - 2\e m \\
&\geq& (1-19\e)m.
\end{eqnarray*}

\end{document}